\documentclass[a4paper,11pt]{amsart}
\usepackage{graphicx}
\usepackage{amsmath}
\usepackage{amssymb}
\usepackage{oldgerm}
\usepackage[all]{xy}

\newcommand{\Hom}{\operatorname{Hom}\nolimits}

\renewcommand{\mod}{\operatorname{mod}\nolimits}

\newcommand{\id}{\operatorname{id}\nolimits}
\newcommand{\Ext}{\operatorname{Ext}\nolimits}

\newcommand{\Z}{\operatorname{\mathbb{Z}}\nolimits}
\newcommand{\Q}{\operatorname{\mathbb{Q}}\nolimits}
\newcommand{\R}{\operatorname{\mathbb{R}}\nolimits}
\newcommand{\N}{\operatorname{\mathbb{N}}\nolimits}

\newcommand{\udim}{\operatorname{\underline{dim}}\nolimits}
\newcommand{\lcm}{\operatorname{lcm}\nolimits}
\newtheorem{theo}{Theorem}[section]

\newtheorem{cor}[theo]{Corollary}

\newtheorem{lemma}[theo]{Lemma}

\newcommand{\ko}{\;\; ,}

\begin{document}
\baselineskip=15pt
\title[Linear recurrence relations]{Linear recurrence relations for cluster variables of affine quivers}
\author{Bernhard Keller and Sarah Scherotzke}
\address{Universit\'e Paris Diderot - Paris~7, UFR de Math\'ematiques, Institut de Math\'ematiques de Jussieu, UMR 7586 du CNRS, Case 7012, B\^atiment Chevaleret, 75205 Paris Cedex 13, France}

\email{keller@math.jussieu.fr, scherotzke@math.jussieu.fr}

 \subjclass[2000]{}

\begin{abstract}We prove that the frieze sequences of cluster variables
associated with the
vertices of an affine quiver satisfy linear recurrence relations. In particular, we obtain
a proof of a recent conjecture by Assem-Reutenauer-Smith.
 \end{abstract}

\maketitle


\section{Introduction}
Caldero and Chapoton noted in \cite{CalderoChapoton06} that one obtains
natural generalizations of Coxeter-Conway's frieze patterns \cite{Coxeter71} \cite{CoxeterConway73a} \cite{CoxeterConway73b} when one constructs
the bipartite belt of the Fomin-Zelevinsky cluster algebra \cite{FominZelevinsky02}
associated with a (connected) acyclic quiver $Q$. Such a generalized frieze pattern
consists of a family of sequences of cluster variables, one sequence
for each vertex of the quiver. For simplicity, we call these sequences the
{\em frieze sequences} associated with the vertices of $Q$.
Recently, they have been studied by
Assem-Reutenauer-Smith \cite{AssemReutenauerSmith09} and
by Assem-Dupont \cite{AssemDupont10} for affine quivers $Q$.
They also appear implicitly in the work of Di Francesco and Kedem,
cf. for example \cite{DiFrancescoKedem10} \cite{DiFrancescoKedem10a}.

Our main motivation in this paper comes from a conjecture formulated by
Assem, Reutenauer and Smith  \cite{AssemReutenauerSmith09}: They
proved that if the frieze sequences associated
with a (valued) quiver $Q$ satisfy linear recurrence relations, then $Q$ is
necessarily affine or Dynkin. They conjectured that conversely, the frieze sequences
associated with a quiver of Dynkin or affine type always satisfy linear
recurrence relations.
For Dynkin quivers, this is immediate from Fomin-Zelevinsky's classification
theorem for the finite-type cluster algebras \cite{FominZelevinsky03}.
In \cite{AssemReutenauerSmith09},
Assem-Reutenauer-Smith gave an ingenious proof
for the affine types $\tilde{A}$ and $\tilde{D}$
as well as for the non simply laced types obtained from these by folding.
For the exceptional affine types, the conjecture remained open.

In this paper, we prove Assem-Reutenauer-Smith's conjecture in full generality
using the representation-theoretic approach to cluster algebras
pioneered in \cite{MarshReinekeZelevinsky03}. More precisely, our
main tool is the categorification of acyclic cluster algebras via cluster categories
(cf. e.g. \cite{Keller08c}) and especially the cluster multiplication
formula of \cite{CalderoKeller06}.
Our method also yields a new proof for $\tilde{A}$ and $\tilde{D}$. It leads to linear recurrence relations which are explicit for the
frieze sequences associated with the extending vertices and which allow
us to conjecture explicit minimal linear recurrence relations for all vertices.

\section*{Acknowledgment}
The first-named author thanks Andr\'e Joyal and Christophe
Reutenauer for their kind hospitality at the Laboratoire de
combinatoire et d'informa\-tique math\'ematique at the
Universit\'e du Qu\'ebec \`a Montr\'eal,  where he first learned
about the problem at the origin of this article. The second-named
author thanks Alastair King for a helpful discussion and La
Fondation Sciences Math\'ematiques de Paris for their financial
support.

\section{Main result and plan of the paper}
Let $Q$ be a finite quiver without oriented cycles. We assume
that its vertices are numbered from $0$ to $n$ in such a way
that each vertex $i$ is a sink in the full subquiver on the
vertices $1$, \ldots, $i$. We introduce a total order on the set $\N \times
Q_0$ by requiring that $(j,i) \le (j', i')$ if we have $j < j' $
or both $j=j'$ and $i\ge i'$ hold.

The {\em generalized frieze pattern associated with $Q$}
is a family of sequences $(X^i_j)_{j\in\N}$ of elements of the
field $\Q(x_0, \ldots , x_n)$, where $i$ runs through the vertices of $Q$.
We recursively define these {\em frieze sequences} as follows: we set
$X^i_0=x_i$ for all vertices $i$ of the quiver. Once $X^{i'}_{j'}$ has been
defined for every pair $(j',i')< (j+1,i)$, we define $X^i_{j+1}$ by the
equality
\[
X^i_{j+1}X^i_{j}=1+\prod_{s\to i} X^s_{j+1}\prod_{i\to d} X^d_j.
\]
Note that the elements of the frieze sequences defined in this way
are cluster variables of the cluster algebra $\mathcal{A}_Q$ associated
with $Q$, cf. \cite{CalderoChapoton06} \cite{Keller08c}.
The aim of this paper is to
show the following result:
\begin{theo}
If $Q$ is an affine quiver, then every frieze sequence $(X^i_j)_{j\in \N}$ satisfies a linear
recurrence relation.
\end{theo}

This confirms the main conjecture of Assem-Reutenauer-Smith's
\cite{AssemReutenauerSmith09}. They proved it
for the case where $Q$ is of type $\tilde A$ or $\tilde D$
(and for the non simply laced types obtained from these
by folding).
We will provide a new
proof for these types and an extension to the exceptional types.
Following \cite{AssemReutenauerSmith09} we show in section~9,
using the folding technique,
that the theorem also holds for affine valued quivers.

Our proof is based on the additive categorification
of the cluster algebra $\mathcal{A}_Q$ by the cluster category of
$Q$ as introduced in \cite{BuanMarshReinekeReitenTodorov06}.
In addition to the cluster category, the main ingredient of our proof
is the Caldero-Chapoton map \cite{CalderoChapoton06}, which takes each object
of the cluster category to an element of the field $Q(x_0,\ldots , x_n)$.
Under this map, the exchange relations used to define the cluster variables
are related to certain pairs of triangles in the cluster category,
called exchange triangles. We will obtain linear recurrence
relations from `generalized exchange relations' obtained via the
Caldero-Chapoton map from `generalized exchange triangles'.

The main steps of the proof are as follows:

{\em Step 1.} We describe the action of the Coxeter transformation
on the root system of an affine quiver.

{\em Step 2. } We show that the frieze sequence associated with a vertex
$i$ of the quiver is the image under the Caldero-Chapoton map of the
$\tau$-orbit of the projective indecomposable module associated with
the vertex~$i$.

{\em Step 3. } We prove the existence of generalized exchange
triangles in the cluster category of an affine quiver using Step 1.

{ \em Step 4.} By Step 2 we can deduce relations between the
frieze sequences associated with vertices of the quiver from the
generalized exchange triangles constructed in Step 3.

{\em Step 5.} The relations between frieze sequences obtained in Step 4 are
either linear recurrence relations or they show that a fieze sequence is a
product or sum of sequences that satisfy a linear recurrence. Hence
all frieze sequences satisfy a linear recurrence relation.

In section~1, we study the action of the Coxeter
transformation $c$ of an affine quiver on the roots corresponding
to preprojective indecomposables. We use this result to
determine, for every affine quiver, the minimal strictly positive
integers $b$ and $m$ such that $c$ satisfies
$c^b=\id+m\langle-, \delta \rangle \delta$, where $ \langle -,-
\rangle$ denotes the Euler form of the quiver. Let us stress that
$b$ is {\em not} the Coxeter number of the associated finite root system.

In section~2, we briefly recall the cluster category of a quiver without
oriented cycles. We introduce the Caldero-Chapoton map from the class of objects
of the cluster category to $\Q(x_0,\ldots, x_n)$ and define exchange
triangles and generalized exchange triangles of the cluster
category. We state a result which describes how a pair of exchange
triangles determines an equation between the images of the objects
appearing in the triangles under the Caldero-Chapoton map. Then we show that the
frieze sequence associated with a vertex $i$ is obtained by applying the
Caldero-Chapoton map to the $\tau$-orbit of the projective indecomposable module
associated with $i$ viewed as an object in the cluster category.

In the third section, we give conditions, in the case of an affine
quiver, for the existence of certain generalized exchange triangles.
We deduce linear recurrence relations from these generalised exchange
triangles, using the results of the previous section.

In the next three section, we show that the main theorem holds for
affine quivers. In doing so we use results of section~1 to show that
the conditions of section~3 are satisfied. The exchange
triangles yield relations between frieze sequences that prove that the
sequences satisfy linear recurrence relations.

In section~6, we prove the main theorem for affine
quivers of type $\tilde D$ and in section~7 for all exceptional
affine quivers. Here the linear recurrence relations are given explicitly
for frieze sequences associated with extending vertices. For all other
frieze sequences, the existence of a linear recurrence is proven by showing
that every sequence associated with a vertex can be written as a
product or a linear combination of sequences satisfying a linear recurrence.

In section~8, we prove the main theorem for affine quivers of
type $\tilde A_{p,q}$. Here the explicit linear recurrence relations are
given only if $p$ equals $q$. Otherwise, the existence of a linear
recurrence relation is shown simultaneously for all frieze sequences
by considering the sequence of vectors in $ \Q(x_0, \ldots ,x_n)^{n+1}$
whose $i$th coordinate is given by the entries of the frieze sequence
associated with the vertex $i$ for all vertices $i$ of the quiver.

In section~9, we extend the main theorem to
valued quivers using the folding technique and in the
final section~10, we conjecture explicit minimal linear
recurrence relations.

\section{On the Coxeter transformation of an affine quiver}

We first fix the notation and recall some basic facts.
We refer to \cite{CrawleyBoevey92} and \cite{AssemSimsonSkowronski06}
for an introduction to quivers and their representations.

Let $Q$ be an affine quiver, i.e.~a quiver whose underlying graph is
an extended simply laced Dynkin diagram $\tilde{\Delta}$. The {\em type of $Q$}
is the diagram $\tilde{\Delta}$ except if we have $\tilde{\Delta}=\tilde{A}_n$,
in which case the type of $Q$ is $\tilde{A}_{p,q}$ where, for a chosen
cyclic orientation of the underlying graph of $Q$, the number of positively
(respectively negatively) oriented arrows equals $p$ (respectively $q$).
We number the vertices of $Q$ from $0$ to $n$ and define the
{\em Euler form} of $Q$ as the bilinear form $\langle, \rangle$ on
$\Z^{n+1}$ such that, for $a$ and $b$ in $\Z^{n+1}$, we have
\[
\langle a,b\rangle = \sum_{i=0}^n a_i b_i - \sum_{i,j=0}^n c_{ji} a_i b_j \ko
\]
where $c_{ij}$ is the number of arrows from $i$ to $j$ in $Q$. The
{\em symmetrized Euler form} $(,)$ is defined by
\[
(a,b)=\langle a,b\rangle+\langle b,a \rangle
\]
for $a$ and $b$ in $\Z^{n+1}$. A {\em root} for $Q$ is a non zero vector
$\alpha$ in $\Z^{n+1}$ such that $(\alpha,\alpha)/2\leq 1$; it is {\em real} if
we have $(\alpha,\alpha)/2=1$ and {\em imaginary} if $(\alpha,\alpha)=0$.
It is {\em positive} if all of its components are positive. The
{\em root system $\Phi$} is the set of all roots. There is a unique root $\delta$
with strictly positive coefficients whose integer multiples form the radical
of the form $(,)$ (cf. Chapter~4 of \cite{CrawleyBoevey92}). A vertex $i$
of $Q$ is an {\em extending vertex} if we have $\delta_i=1$. If $\alpha$ is a
real root, the {\em reflection} at $\alpha$ is the automorphism $s_\alpha$ of
$\Z^{n+1}$ defined by
\[
s_\alpha(x) = x - (\alpha,x) \alpha.
\]
For each vertex $i$, the {\em simple root $\alpha_i$} is the
$(i+1)$th vector of the standard basis of $\Z^{n+1}$. Let us number
the vertices in such a way that each vertex $i$ is a sink of the
full subquiver of $Q$ on the vertices $0$, \ldots, $i$. Using this
ordering, we define the {\em Coxeter transformation} of $Q$ to
be the composition
\[
c= s_{\alpha_0} s_{\alpha_1}  \cdots  s_{\alpha_n}.
\]
We have
\[
\langle x,y \rangle = - \langle y, cx \rangle
\]
for all $x$ and $y$ in $\Z^{n+1}$.

Let $k$ be an algebraically closed field and $kQ$ the path algebra of $Q$
over $k$. Let $\mod kQ$ be the category of $k$-finite-dimensional right
$kQ$-modules. For a vertex $i$ of $Q$, we denote the simple module
supported at $i$ by $S_i$, its projective cover by $P_i$ and its
injective hull by $I_i$. The map taking a module $M$ to its
{\em dimension vector}
\[
\udim M= (\dim \Hom(P_i, M))_{i=0\ldots n}
\]
induces an isomorphism from the Grothendieck group of $\mod kQ$ to
$\Z^{n+1}$. By Kac's theorem, the dimension vectors of the indecomposable
modules are precisely the positive roots. For two modules $L$ and $M$, we have
\[
\langle \udim L, \udim M\rangle = \dim \Hom(L,M)- \dim \Ext^1(L,M).
\]
For an indecomposable non injective module $M$, we have
\[
c^{-1} \udim M = \udim \tau_m^{-1} M \ko
\]
where $\tau$ is the Auslander-Reiten translation
of the module category $\mod kQ$.

\begin{theo}\label{delta}
There exist a strictly positive integer $b$ and a non zero
integer $m$ such that $c^b= \id - m \langle -, \delta \rangle
\delta$. The integer $b$ is a multiple of the width of the tubes
in the Auslander-Reiten quiver of $Q$.
\begin{itemize}
\item[(1)] For $Q$ of type $\tilde E_t$, the minimal $b$ is given by $b=6$
for $t=6$; $b=12$ for $t=7$ and $b=30$ for $t=8$. In all those cases
$m$ is equal to 1.

\item[(2)] For $Q$ of type $ \tilde D_n$, we have for even $n$ that $b=n-2$
and $m=1$; if $n$ is odd, we have $b=2n-4$ and $m=2$.

\item[(3)] For $Q$ of type $ \tilde A_{p,q}$, the minimal $b$ is the least common
multiple of $p$ and $q$ and $m$ is the order of the class of $q$ in
the additive group $\Z / (p+q) \Z$.
\end{itemize}
\end{theo}
We will give a uniform interpretation of the integer $b$ in
Lemma~\ref{lemma:interpret-b} below. Let us stress that, contrary to
a common misconception, it is {\em not} the Coxeter number of the
corresponding finite root system.

\begin{proof}
The automorphism induced by $c$ permutes the elements of the image
of $\Phi\cup \{0\}$ in $\Z^{n+1}/\Z \delta$. This image is finite
(see \cite[7]{CrawleyBoevey92}) and generates $\Z^{n+1}/\Z \delta $. Therefore there exists
 a strictly positive integer $b$ such that $c^b$ induces the
identity on $\Z^{n+1}/\Z\delta$. It follows that there is a linear
form $f: \Z^{n+1} \to \Z$ such that $c^b-\id$ is equal to $\langle
f,- \rangle\delta$. In order to show that $f$ is a multiple of
$\langle -, \delta \rangle$, as $\langle -, \delta\rangle$ is
primitive, it is sufficient to show that $f$ vanishes on the
kernel of $\langle -, \delta \rangle$. By \cite[7]{CrawleyBoevey92} the kernel
is generated by the dimension vectors of the regular modules.
Clearly it is enough to verify that $f$ vanishes on the dimension
vectors of the regular simple modules. Let $M$ be such a module.
If $M$ lies in a homogenous tube, its dimension vector is $\delta$
and $f(\delta)$ vanishes by construction. Let us therefore assume
that $M$ is in an exceptional tube of width $s>1$. The dimension
vectors of $\udim M$, $\udim \tau M, \ldots, \udim \tau^{s-1} M$
are non-zero and have sum $\delta$. It follows that they are real
roots and two by two distinct. Moreover the difference between two
of these vectors is not a non-zero multiple of $\delta$. Therefore
their images in $\Z^{n+1}/ \Z \delta$ are pairwise distinct. We
must therefore have $c^b(\udim M)=\udim M$ and $f( \udim M)$
vanishes. This argument also shows that the widths of the tubes
divide $b$.

(1) The values of $b$ and $m$ for the exceptional quivers can be
verified by direct computation using for example the cluster mutation applet \cite{KellerQuiverMutationApplet}.

For the other cases, we need a more detailed description of the
roots and of the Coxeter transformation. Let $Q'$ be the Dynkin
quiver obtained from $Q$ by deleting the extending vertex $0$ and
all arrows adjacent to it. Let $\alpha_1, \ldots, \alpha_n$ be the
root basis of $Q'$ consisting of the dimension vectors of the
simples and let $\theta$ be the highest root of $Q'$. Via the
inclusion of $Q'$ into $Q$ we identify the roots of $Q'$ with
their image in $\Z^{n+1}$. Then the dimension vector of the simple
at the vertex $0$ is $\alpha_0=\delta-\theta$.

(2) We choose the following labeling and orientation on $\tilde
D_n$:

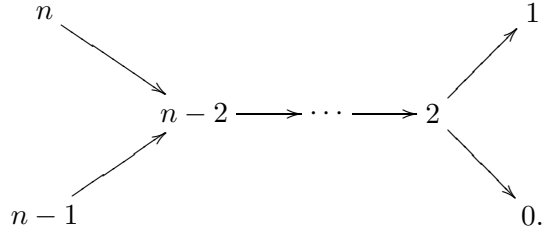
\[\xymatrix{ n \ar[rd]& &&& 1 \\& n-2  \ar[r]& \ar[r] \cdots & 2 \ar[ru] \ar[rd]& \\
n-1 \ar[ru]  &&&& 0. }\]Let $e_1, \ldots, e_n$ be the vectors in $\R^{n+1}$ defined by
$$ \alpha_i=e_i-e_{i+1} \mbox { for }1 \le i \le n-1 \mbox{ and } \alpha_n=e_n+e_{n-1}.$$ Then if we extend the form $(-,-)$ to $\R^{n+1}$, we have $(e_i,e_j)=\delta_{i,j}$ and $(e_i, \delta)=0$.
Furthermore $\theta$ equals $e_1+e_2$ and $\alpha_0$ equals
$\delta-e_1-e_2$. The reflections $s_{\alpha_i}$ for $1 \le i \le
n-1$ act as the transposition of $e_i$ and $e_{i+1}$. The
reflection at $\alpha_n$ maps $e_n$ to $-e_{n-1}$ and $e_{n-1} $
to $-e_n$. The reflection at $\alpha_0$ is given by $e_1 \mapsto
-e_2+\delta $ and $e_2 \mapsto -e_1 +\delta$.

We see that $c=s_{\alpha_0} \cdots s_{\alpha_{n}}$ acts up to
multiples of $\delta$ as the $(n-2)$-cycle on $e_2, \ldots,
e_{n-1} $ and inverses the sign of $e_1$ and $e_n$.  So $c$ maps
$e_i$ to $e_{i+1}$ for $2 \le i < n-1$ and $e_{n-1}$ to
$e_2-\delta$ and $e_1$ to $-e_1+\delta$ and $e_n$ to $- e_n$. Then
$c^{n-2}$ corresponds to the action \[
e_1 \mapsto \left \{ \begin{array}{ll} e_1 & \mbox{ if $n$ is even }\\
                              -e_1+\delta & \mbox{ else.}
\end{array} \right. \] and $e_i \mapsto e_i-\delta$ for $2 \le i \le n-1$ and $e_n
\mapsto (-1)^{n-2}e_n$. Therefore we have that $b$ equals $n-2$ if $n$
is even and $b$ equals $2n-4$ if $n$ is odd. We see that $c^{n-2}
$ maps $ \alpha_n$ to $\alpha_n-\delta$ if $n$ is even and it maps
$\alpha_n$ to $\alpha_{n-1} -\delta $ if $n $ is odd. As $\langle
\alpha_n , \delta \rangle =1 $, this shows that $m$ is equal to
$1$ for $n$ even and $m$ is equal to $2$ if $n$ is odd.

(3) We consider the case $\tilde A_{p,q}$ with $q \le p$. We choose
the following orientation and labeling on $\tilde A_{p,q}$:

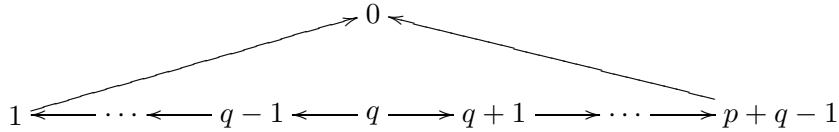
\[ \xymatrix{&&& 0 &&&\\ 1 \ar[urrr] &  \cdots \ar[l] & q-1 \ar[l] & q \ar[r] \ar[l] & q+1 \ar[r] & \cdots \ar[r] & p+q-1 \ar[ulll]}\]
Let $E$ be a real vector space with basis $e_1, \ldots , e_{n+1},
d$. We endow $E$ with a symmetric bilinear form $(-,-)$ such that
$(e_i, e_j)=\delta_{i,j} $ and $(d, e_i)=(d, d)=0$.  We have an
isometric embedding of $\Z^{n+1} $ into $E$ taking
\[\alpha_i \mapsto e_i-e_{i+1} \mbox{ for } 1\le i \le n \mbox{ and } \delta \mapsto d. \]Then $\theta$ is mapped to $e_1-e_{n+1}$ and $\alpha_0 $ is mapped to $d-e_1+e_{n+1}$. From now on, we identify $\Z^{n+1}$ with a subset of $E$ using this embedding.  The reflection
$s_{\alpha_i}$ acts as the transposition of $e_i$ and $e_{i+1}$ for
$1 \le i \le n$. The reflection at $\alpha_0$ maps $e_1$ to
$e_{n+1}+\delta $ and $e_{n+1}$ to $e_1 -\delta$. Then $c$ is given
by the product $s_{\alpha_0} \cdots
s_{\alpha_{q-1}}s_{\alpha_{p+q-1}}\cdots s_{\alpha_q}$. The action
of $c$ is up to multiples of $\delta$ the product of the $q$-cycle
on $e_1, \ldots ,e_q$ and the $p$-cycle on $e_{p+q},\ldots,
e_{q+1}$. More concretely, we have $e_i \mapsto e_{i+1}$ for $1\le i
\le q-1$, $e_q \mapsto e_1-\delta $ and $e_i \mapsto e_{i-1} $ for
$q+2 \le i \le p+q$ and $e_{q+1} \mapsto e_{p+q}+\delta$. Then
$c^{\lcm(p,q)}$ corresponds to the action $e_i \mapsto
e_i-(\lcm(p,q)/q) \delta$ for $1 \le i \le q$ and $e_i \mapsto
e_i+(\lcm(p,q)/p) \delta$ for $q+1 \le i \le p+q$. Therefore we have
$b=\lcm(p,q)$. We verify that $c^{\lcm(p,q)} $ maps $ \alpha_q$ to
$\alpha_q-(\lcm(p,q)/p+\lcm(p,q)/ q) \delta$. As $\langle \alpha_q ,
\delta \rangle =1 $, this shows that $m$ is equal to
$\lcm(p,q)/p+\lcm(p,q)/ q$ which is the order of the class of $q$ in
$\Z/ (p+q) \Z$.
\end{proof}
We have more information when $Q$ is of type $\tilde E_t$ for
$t=6$ and $t=7$. Then, for each $i\in Q_0$ there are positive integers
$k_i$ such that $k_i \delta_i=b$. These $k_i$ satisfy $c^{k_i} \udim
P_i= \udim P_i- \delta$.

Let us give a uniform interpretation of the integer $b$ of the Theorem.
We use the notations of the above proof.
Let $c'$ denote the Coxeter transformation of the Dynkin quiver $Q'$.
Let $\bar c$ be the automorphism on $\Z^{n+1}/\Z \delta$ induced
by $c$.
\begin{lemma} \label{lemma:interpret-b}
The automorphism $\bar c$ equals $s_{\theta} c'$. Hence $b$ equals
the order of the element $s_{\theta} c'$ in the Weyl group of $Q'$.
\end{lemma}
\begin{proof}
The embedding of the root system of $Q'$ in $Q$ given in the proof of
\ref{delta} yields an embedding of the Weyl group of $Q'$ into the
Weyl group of $Q$ such that every reflection at a root of $Q'$ fixes
$\delta$. Hence $c$ equals $s_{\delta-\theta}c'$. We can write every
element $y \in \Z^{n+1} $ as a linear combination of the roots of
$Q'$ and $\delta$. Let $y =j+t\delta$ where $j$ is a linear
combination of the roots of $Q'$ and $t \in \Z$. Then
$$s_{\delta-\theta}(j+t\delta)=j-(\theta,j)\theta+(t+(\theta,j))\delta=s_{\theta}(j)+(t+(\theta,j))\delta.$$
Therefore the action of $s_{\delta-\theta}$ modulo $\delta$ equals
the action of $s_{\theta}$. As $c'$ fixes $\delta$, we have $\bar
c=s_{\theta} c'$ and $b$ is the order of $s_{\theta} c'$.
\end{proof}
We denote by $\sigma$ the automorphism on $\tilde D_n$ with $\sigma
1=0$, $\sigma 0=1 $ and $\sigma n=n-1$, $\sigma (n-1)=n$ and
$\sigma$ fixes all other vertices of $\tilde D_n$. Recall that the
extending vertices of $\tilde D_n$ are precisely $0$, $1$, $n-1$ and
$n$.

\begin{lemma}\label{relation}
(a) Let $Q$ be of type $\tilde D_n$. Suppose that $n$ is odd and $i$
is an extending vertex of $Q$. Then $c^{n-2}(\udim P_i)= \udim
P_{\sigma i}-\delta$.

(b) For every vertex $i$ of $\tilde A_{p,q}$ we have
$c^{l_i}(\udim P_{i-q})=\udim P_i +\delta $, where $l_i=i-q$ for $
0\le i \le q$ and $l_i=\mbox{max}\{q-i, -q\}$ for $q < i$.
\end{lemma}
\begin{proof}
(a) We have $\langle \udim P_i, \delta \rangle =\delta_i$, which
equals one as $i$ is extending. By the proof of \ref{delta}, we
have $c^{n-2}(\udim P_n) =c^{n-2}(\alpha_n)= \alpha_{n-1}
-\delta=\udim P_{n-1}-\delta$. If we apply $c^{n-2} $ to this
equation, we obtain $\udim P_n -2\delta =c^{2n-4}(\udim
P_n)=c^{n-2}(\udim P_{n-1})-\delta$ and thus $\udim P_n -\delta
=c^{n-2}(\udim P_{n-1})$.

Furthermore $c^{n-2}(\udim P_1)= c^{n-2}(\sum_{i=1}^n
\alpha_i)=c^{n-2}( e_1+e_{n-1})=-e_1+\delta +e_{n-1}-\delta=
e_2+e_{n-1}+\delta -\theta - \delta = \sum_{i=2}^n \alpha_i
+\alpha_0=\udim P_0-\delta$. Analogously to the first case, we
apply $c^{n-2} $ to the equation and obtain $\udim P_1 -\delta$
equals $c^{n-2}(\udim P_{0})$.

(b) We first assume that $i$ satisfies $ 0< i < q$. By the proof
of \ref{delta}, we have $c^{q-i}(\udim P_i)=c^{q-i}(\sum^q_{l=i}
\alpha_l)=c^{q-i}(e_i-e_{q+1})=e_q-e_{p+i+1}-\delta=\udim
P_{p+i}-\delta$. For $i=q$ we have $\udim P_q= e_q-e_{q+1}=\udim
P_0-\delta$ and for $i=0$ we have $c^q(\udim P_0)=c^q(
e_q-e_{q+1}+\delta)= e_q-\delta -e_{p+1}-\delta+\delta=\udim
P_p-\delta$.

Let $q+1 \le i \le 2q$, then $c^{q-i}(\udim
P_i)=c^{q-i}(e_q-e_{i+1})=e_{i-q}-e_{q+1}-\delta=\udim
P_{i-q}-\delta$.

Let finally $ 2q \le i \le p+q-1$ if $p>q$, then $c^q(\udim
P_i)=c^q (e_q-e_{i+1})=e_q -e_{i-q+1}-\delta= \udim
P_{i-q}-\delta$, which finishes the proof.
\end{proof}

\section{Frieze sequences of cluster variables}
Let $\mathcal{F}$ denote the field $\Q(x_0,\ldots, x_n)$.
A sequence $(a_j)_{j\in\N}$ of elements in $\mathcal{F}$ {\em
satisfies a linear recurrence} if for some integer $s\ge 1$, there
exist elements $\alpha_0,\ldots, \alpha_{s-1}$ in $\mathcal{F}$ such
that for all $j\in \N$, one has $a_{j+s}=\alpha_0a_{j}+\ldots
\alpha_{s-1} a_{j+s-1}$. Equivalently, the generating series
\[
\sum_{j\in \N} a_j \lambda^j
\]
in $\mathcal{F}[[\lambda]]$ is rational and its denominator is a multiple
of the polynomial
$P(\lambda)=\lambda^s - \alpha_{s-1} \lambda^{s-1} - \ldots - \alpha_0$.
We say that the polynomial {\em annihilates} the sequence.
\begin{lemma}\label{ProductSum}
(a) Let $(a_j)_{j\in \N}$ and $(b_j)_{j\in \N}$ be two sequences in
$\mathcal{F}$ that satisfy a linear recurrence relation. Then the
sequences $(a_j+b_j)_{j\in\N}$ and $(a_jb_j)_{j\in\N}$ satisfy a
linear recurrence relation.

(b) Let $m \ge 1$ be an integer and for each $1 \le i \le m$, let
$(a^i_j)_{j\in \N}$ be a sequence in $\mathcal{F}$. We consider the
sequence of vectors $(v_j)_{j\in \N}$ defined by $v_j=(a^1_j,
\ldots, a^m_j)^t$ for all $j\in\N$. Suppose there exist $m\times m$
matrices $A_0, \ldots ,A_{s-1}$ over $\mathcal{F}$ such that for
every $j \in \N$ we have $v_{j+s}=A_0v_{j}+\ldots +A_{s-1}
v_{j+s-1}$. Then each sequence $(a^i_j)_{j\in \N}$ satisfies a
linear recurrence.
\end{lemma}
\begin{proof} We refer to \cite{BerstelReutenauer10} for complete proofs
of these fundamental facts. Let us record however, that if the two series
are annihilated by polynomials $P$ and $Q$, then their sum is
annihilated by $PQ$ and their Hadamard product $(a_j b_j)_{j\in\N}$
by the characteristic polynomial of $C_P \otimes_\mathcal{F} C_Q$, where $C_P$ is
the companion matrix of $P$. In b), the sequences are annihilated
by the determinant of the matrix
$\lambda^s - \lambda^{s-1} A_{s-1} - \ldots - \lambda A_1 - A_0$.
\end{proof}

We refer to \cite{Keller08c} for an introduction to the links
between cluster algebras and quiver representations which
we now briefly recall.
Let $\mathcal{D}_Q$ denote the bounded derived category of
$kQ$-modules. It is a triangulated category and
we denote its suspension functor by $\Sigma:\mathcal{D}_Q\to
\mathcal{D}_Q$. As $kQ$ has finite global dimension,
Auslander-Reiten triangles exist in $\mathcal{D}_Q$ by
\cite[1.4]{Happel91a}. We denote
the Auslander-Reiten translation of $\mathcal{D}_Q$ by $\tau$. On
the non projective modules, it coincides with the Auslander-Reiten
translation of $\mod kQ$. The
{\em cluster category} \cite{BuanMarshReinekeReitenTodorov06}
\[ \mathcal{C}_Q=\mathcal{D}_Q/(\tau^{-1} \Sigma)^{\Z}\] is the orbit
category of $\mathcal{D}_Q$ under the action of the cyclic group
generated by $\tau^{-1} \Sigma$. One can show \cite{Keller05} that
$\mathcal{C}_Q$ admits a canonical structure of triangulated
category such that the projection functor $\pi:\mathcal{D}_Q \to
\mathcal{C}_Q$ becomes a triangle functor.

From now on, we assume that the field $k$ has characteristic $0$. We
refer to \cite{CalderoKeller06} for the definition of the
Caldero-Chapoton \cite{CalderoChapoton06} map $L \mapsto X_L$ from the set of
isomorphism classes of objects $L$ of $\mathcal{C}_Q$ to the field
$\mathcal{F}$. We have $X_{\tau P_i}= x_i$ for all vertices $i$ of
$Q$ and $X_{M\oplus N}= X_M X_N$ for all objects $M$ and $N$ of
$\mathcal{C}_Q$.
We call an object $M$ in $\mathcal{C}_Q$ {\em rigid} if it has no
self-extensions, that is if the space $\Ext^1_{\mathcal{C}_Q}(M,M)$ vanishes.

\begin{theo}[\cite{CalderoKeller06}] \label{acyclic-main-thm}
\begin{itemize}
\item[a)] The map $L \mapsto X_L$ induces a bijection from the set
of isomorphism classes of rigid indecomposables of the cluster
category $\mathcal{C}_Q$ onto the set of cluster variables of the
cluster algebra $\mathcal{A}_Q$. \item[b)] If $L$ and $M$ are
indecomposables such that the space $\Ext^1(L,M)$ is
one-dimensional, then we have the generalized exchange relation
\begin{equation}\label{eq:acyclic-main-thm}
X_L X_M = X_E + X_{E'}
\end{equation}
where $E$ and $E'$ are the middle terms of `the' non split triangles
\[
\xymatrix{L \ar[r] & E \ar[r] & M \ar[r] & \Sigma L} \mbox{ and }
\xymatrix{M \ar[r] & E' \ar[r] & L \ar[r] & \Sigma M}.
\]
\end{itemize}
\end{theo}
Let $L$ and $M$ be
two indecomposable objects in the cluster category such that
$\Ext^1_{\mathcal{C}_Q} (M,L)$ is one dimensional. If both $L$ and
$M$ are rigid, then so are $E$ and $E'$ and the sequence
(\ref{eq:acyclic-main-thm}) is an exchange relation of the cluster
algebra $\mathcal{A}_Q$ associated with $Q$. Therefore in this case, we call the
triangles in (\ref{acyclic-main-thm}) {\em exchange triangles}. If $L$
or $M$ is not rigid, we call them {\em generalized exchange
triangles}.
\begin{cor}
For each vertex $i$ of $Q_0$ and each $j$ in $\N$, we have
$X^i_j=X_{\tau^{-j+1} P_i}$.
\end{cor}
\begin{proof}By the definition, the initial variables $x_0, \ldots, x_n$ are
the images under the Caldero-Chapoton map of $\tau P_0, \ldots , \tau
P_n$. The Auslander-Reiten component of $\mathcal{D}_Q$ containing
the projective indecomposable modules is isomorphic to $\Z Q$, where
the vertex $(j, i)$ of $\Z Q$ corresponds to the isomorphism class
of $\tau^{-j+1} P_i$ for all vertices $i$ of $Q$ and $j \in \Z$. To
prove the statement, we use induction on the ordered set $\N\times
\Q_0$. The claim holds for all vertices of $Q$ and $j=0$. Now let
$(j,w)$ be a vertex of $\N \times \Q_0$ such that $j >0$. By the
induction hypothesis, we have $X_{\tau^{-j+2} P_i}= X^i_{j-1}$ for
all vertices $i$ of the quiver and $X_{\tau^{-j+1} P_i}= X^i_j$ for
all $i> w$. We consider the Auslander-Reiten triangle ending in
$\tau^{-j+1} P_w$

\[\tau^{-j+2} P_w \to (\bigoplus_{s\to w } \tau^{-j+2}P_s) \oplus  (\bigoplus_{w \to d} \tau^{-j+1}P_d)  \to
\tau^{-j+1} P_w \to \Sigma \tau^{-j+2} P_w. \]The three terms of
this triangle are rigid and the space of extensions of $\tau^{-1}
P_w$ by $P_w$ is one-dimensional. By \ref{acyclic-main-thm} part b),
this yields the exchange relation
$$X_{\tau^{-j+2} P_w} X_{\tau^{-j+1}P_w} = 1+ \prod_{w\to s}
X_{\tau^{-j+2}P_s}\prod_{d \to w} X_{\tau^{-j+1}P_d}.$$ By the
induction hypothesis, this translates into the relation
$X^w_{j-1}X_{\tau^{-j+1} P_w}=1+\prod_{w\to s } X^s_{j-1}\prod_{d
\to w} X^d_j$. Therefore $X_{\tau^{-j+1}P_w}$ equals $X^w_{j}$,
which proves the statement.
\end{proof}

\section{Generalized exchange triangles in the cluster category}
Let $Q$ be an affine quiver. In this section, we
construct some generalized exchange triangles in the cluster category
$\mathcal{C}_Q$.
\begin{lemma}\label{existence}
Let $L$ and $N$ be two indecomposable preprojective $kQ$-modules of
defect minus one satisfying the equation $\udim L =\udim N +
\delta$. Then, for every regular simple $kQ$-module $M$ of dimension
vector $\delta$, there exists an exact sequence
\[ 0\to N \to L \to M \to 0 \]
and $\dim_k \Ext_{kQ}^1( M,N)=1$.
\end{lemma}
\begin{proof}
As $N$ has defect minus one, we have
\begin{align*}
-1 &= \langle \delta , \udim N \rangle= \dim \Hom(M,N)-\dim \Ext_{kQ}^1(M, N) \\
    & = - \dim \Ext_{kQ}^1(M, N) .
\end{align*}
By the assumption, we have $\udim L= \udim N
+\delta$ and therefore $1= \langle\udim L, \delta \rangle=\dim \Hom
( L, M)$ as $\Ext_{kQ}^1(L,M)$ vanishes.  Since $M$ is regular
simple, every submodule of $M$ that is not equal to $M$ is
preprojective. Every submodule of $L$ is preprojective hence of
defect at most $-1$. Thus, every quotient of $L$ is of defect $\ge
0$. Since the proper submodules of $M$ are preprojective, every non
zero map from $L$ to $M$ is surjective. The kernel of such a map has
defect $-1$ and is preprojective. Therefore the kernel is
indecomposable and its dimension vector equals $\udim N$. Any
preprojective indecomposable module is determined by its dimension
vector. Thus, the kernel of every non zero map is isomorphic to $N$.
This proves the existence of the exact sequence.
\end{proof}

\begin{lemma}\label{ext}
Let $N$ and $M $ be two $kQ$-modules. Then we have a canonical
isomorphism $\Ext^1_{\mathcal{C}_Q} (M,N)\cong \Ext_{kQ}^1(M, N)
\oplus D \Ext_{kQ}^1(N,M)$.
\end{lemma}
\begin{proof} This is Proposition 1.7 c) of
\cite{BuanMarshReinekeReitenTodorov06}.
\end{proof}

\begin{theo}\label{deltaextension}
Let $i\in Q_0$ be an extending vertex and suppose there is a
positive integer $b$ such that $P_i$ satisfies the equation $\udim
\tau^{-b} P_i= \udim P_i + \delta$. Then, for every regular simple
$kQ$-module $M$ of dimension vector $\delta$, there exist
generalized exchange triangles in $\mathcal{C}_{Q}$
\[
P_i \to \tau^{-b} P_i \to M \to \Sigma P_i  \quad\mbox{and}\quad
M \to \tau^{b} P_i \to P_i \to \Sigma M .
\]
\end{theo}
\begin{proof} The defect of $P_i$ is $\langle \delta, \udim P_i\rangle =-\delta_i$, which equals $-1$ since $i$ is an extending vertex. Therefore, the defect of $\tau^{-b} P_i$ also equals
$-1$ and the existence of the first  triangle follows
from \ref{existence}. If we rotate the first triangle, we obtain a
 triangle $ \Sigma^{-1} M \to P_i \to \tau^{-b} P_i \to
M$. If we apply $\tau^{b}$ to it and use the fact that $\Sigma^{-1}
M\cong \tau^{-1} M \cong M$ in $\mathcal{C}_Q$, we get the second
 triangle. By \ref{existence} and \ref{ext} the vector
space $\Ext_{\mathcal{C}_Q}^1(M, P_i)$ is one-dimensional.
\end{proof}

Note that no indecomposable module with dimension vector $\delta$ is
rigid.
\begin{lemma}\cite[3.14]{Dupont08} \label{lemma:x-delta}
Let $N$ and $M$ be two regular simple $kQ$-modules whose dimension
vectors equal $\delta$. Then $X_M$ equals $X_N$.
\end{lemma}
We set $X_{\delta}=X_M$ for any regular simple module $M$ with
dimension vector $\delta$. By the previous Lemma, $X_{\delta}$ does
not depend on the choice of $M$.

\begin{theo}\label{extending}
Let $i\in Q_0$ be an extending vertex and suppose that there is a
positive integer $b$ such that $P_i$ satisfies the equation $\udim
\tau^{-b} P_i= \udim P_i + \delta$. Then the frieze sequence $(X^i_j)_{j
\in \Z}$ satisfies the linear recurrence relation $ X_{\delta}
X^i_j= X^i_{j-b} + X^i_{j+b}$ for all $ j \in \Z$.
\end{theo}
\begin{proof}
Applying $\tau^{-j}$ to the generalized exchange triangles of
\ref{deltaextension} gives new generalized exchange triangles of the
form
\[  \tau^{-j} P_i \to \tau^{-j-b} P_i \to M \to \tau^{-j} \Sigma P_i
\quad\mbox{and}\quad
M \to \tau^{b-j} P_i \to \tau^{-j} P_i \to  \Sigma M
\]
since $\tau M$ is isomorphic to $M$. These generalized exchange triangles
yield the linear recurrence relation $ X_{\delta}
X^i_j=X^i_{j-b}+X^i_{j+b} $ for all $j\ge b$ by
\ref{acyclic-main-thm} b).
\end{proof}

\section{Type $\tilde D$}
Let $Q$ be of type $\tilde D_n$.
We use the same orientation and labeling of $\tilde D_n$ as in the
proof of \ref{delta}.
\begin{theo} Let $n$ be even and let $i$ be an extending vertex of $Q$.
Then the frieze sequence $(X^i_j)_{j \in \Z}$ satisfies the linear
recurrence relation $ X_{\delta} X^i_j= X^i_{j-n+2} + X^i_{j+n-2}$
for all $ j \ge n-2$.
\end{theo}
\begin{proof} This result follows immediately from \ref{delta} and \ref{extending}.
\end{proof}
\begin{theo}\label{Dn}Suppose that $n$ is odd and $i$ is an
extending vertex  of $Q$.

a) For every regular simple $kQ$-module $M$ with
dimension vector $\delta$, there exist generalized exchange
triangles
\[
P_i \to \tau^{2-n} P_{\sigma i} \to M \to \Sigma P_i
\quad\mbox{ and }\quad
M \to \tau^{n-2} P_{\sigma i} \to P_i \to \Sigma M .
\]

b) The frieze sequence $(X^i_j)_{j \in \Z}$ satisfies the linear
recurrence relation
\[
X^2_{\delta} X^i_j= 2X^i_{j} +X^i_{j-4+2n} +
X^i_{j-2n+4}
\]
for all $ j \ge 2n-4$.
\end{theo}
\begin{proof} a) Using \ref{existence} and \ref{relation} there exist
triangles
\[ P_i \to \tau^{2-n} P_{\sigma i} \to M \to \Sigma P_i \]
and
\[P_{\sigma i} \to \tau^{2-n} P_{i} \to M \to \Sigma P_{\sigma i} .\]

Rotating the second triangle, we get a
triangle

\[ \Sigma^{-1} M \to P_{\sigma i} \to \tau^{2-n} P_i \to M .\]

If we apply $\tau^{n-2} $ to it and use the fact that $M\cong
\Sigma^{-1} M$ in $\mathcal{C}_Q$ and $M$ is $\tau$-periodic of
period one, we get a  triangle in $\mathcal{C}_Q$ of
the form

\[ M \to \tau^{n-2} P_{\sigma i} \to P_i \to M .\]
By \ref{ext} and \ref{existence}, these are generalized exchange
triangles.

b) As in the proof of \ref{extending} we can apply powers of
$\tau$ to the  triangles of a) and we get the
 triangles
\[ \tau^{-j} P_i \to \tau^{2-n-j} P_{\sigma i} \to M \to \tau^{-j}\Sigma P_i \]
and
\[ M \to \tau^{n-2-j} P_{\sigma i} \to \tau^{-j} P_i \to M \]
for all $j \in \Z$. By \ref{extending} these
triangles are generalized exchange triangles and we obtain the
relations $X_{\delta}X^i_j=X^{\sigma i}_{n-2+j}+X^{\sigma
i}_{2-n+j}$ and $X_{\delta}X^{\sigma i}_j=X^{i}_{n-2+j}+X^{
i}_{2-n+j}$. Multiplying the first equation with $X_{\delta}$ and
substituting using the second equation gives the stated recurrence
relation.
\end{proof}
Thus we have obtained linear recurrence relations for the
frieze sequences associated with all extending vertices of the quiver
$Q$ of type $\tilde D_n$. Using Auslander-Reiten triangles we will now deduce the
existence of linear recurrence relations for the frieze sequences
associated with neighbours of extending vertices. There is an
Auslander-Reiten triangle
\[ P_n \to P_{n-2} \to \tau^{-1} P_n \to \Sigma P_n. \]
This gives the recurrence relation for the vertex $n-2$.
Similarly, using the Auslander-Reiten triangle

\[ \tau^{-1} P_1 \to P_2 \to P_1 \to \Sigma \tau^{-1} P_1 \]

we obtain the recurrence relation for the vertex $2$. For the
vertex $n-3$ we will use the following exchange triangles

\[ P_{n-2} \to P_{n-3} \to S_{n-3} \to \Sigma P_{n-2}\]
\[ S_{n-3} \to P_{n} \oplus P_{n-1} \to P_{n-2} \to  \Sigma S_{n-3}\]

and for $2 < i < n-3$ we will use the exchange triangles

\[ P_{i+1} \to P_i \to S_i \to \Sigma P_{i+1}\]
\[ S_i \to P_{i+2} \to P_{i+1} \to  \Sigma S_i.\]

These are indeed exchange triangles since we have
\begin{align*}
-1 &=\langle \alpha_{i}, \alpha_{i+1} \rangle = \langle \alpha_i, \sum_{t=i+1}^n \alpha_t
\rangle =\langle \udim S_i , \udim P_{i+1} \rangle \\
&= \dim \Hom(S_i,P_{i+1})-\dim \Ext_{kQ}^1(S_i, P_{i+1}) = -\dim_k
\Ext_{kQ}^1(S_i, P_{i+1})
\end{align*}
for all $ 2<i< n-1$. We therefore
obtain the relations
\[
X^{n-3}_j=X_j^{n-2}X_{\tau^j S_{n-3}}-X^n_j X^{n-1}_j
\]
for all $j \in \Z$ and
\[
X^i_j=X_j^{i+1}X_{\tau^j S_{i}}-X^{i+2}_j
\]
for all $j \in \Z$ and $2<i < n-2$.

Note that the $S_i$ are regular modules for $2 <i < n-3$ lying in
the exceptional tube of length $n-2$ (cf. the tables at the end of
\cite{DlabRingel74a}). Therefore the corresponding frieze sequence
$X_{\tau^j S_i}$ is periodic. We now use descending induction on
the vertices: we can recover linear recursion formulas for the
frieze sequence associated to a vertex $i$ with $ 3<i< n-1$ from the
linear recursion formulas of sequences associated to vertices $i'
> i$ and the periodic sequences $X_{\tau^j S_i}$.

\section{The exceptional types}
\label{s:exceptional-types}
Let $Q$ be of type $\tilde E_t$ for $t\in \{ 6,7,8 \}$. We will use
the following labeling and orientation:
\[\mbox{ for }\tilde E_6: \ \ \ \xymatrix{1 \ar[r] & 2\ar[r]&7& 6 \ar[l] & 5 \ar[l]  \\
&&4 \ar[u]&& \\
&&3 \ar[u] &&},\] the vector $\delta$ is given by
\[\xymatrix{1 \ar[r] & 2\ar[r]&3& 2 \ar[l] & 1 \ar[l]  \\
&&2 \ar[u]&& \\
&&1 \ar[u] &&};\]
\[\mbox{ for } \tilde E_7 : \ \ \ \xymatrix{1 \ar[r] & 2\ar[r]&3 \ar[r] & 8 & 6 \ar[l] & 5 \ar[l] &  4 \ar[l] \\
&&&7 \ar[u]&&& },\]
the vector $\delta$ is given by \[\xymatrix{1 \ar[r] & 2\ar[r]&3 \ar[r] & 4 & 3 \ar[l] & 2 \ar[l] &  1 \ar[l] \\
&&&2 \ar[u]&&& };\]

\[\mbox{ for }\tilde E_8:\ \ \  \xymatrix{1 \ar[r] & 2\ar[r]&3 \ar[r] & 4 \ar[r]& 5 \ar[r] &9& 8\ar[l] & 7 \ar[l]  \\
&&&&&6 \ar[u]&& },\] the vector $\delta$ is given by
\[ \xymatrix{1 \ar[r] & 2\ar[r]&3 \ar[r] & 4 \ar[r]& 5 \ar[r] &6& 4\ar[l] & 2 \ar[l]  \\
&&&&& 3 \ar[u]&& }.\]

\begin{theo} Let $i$ be an extending vertex of $Q$. Then the frieze
sequence $(X^i_j)_{j \in \Z}$ satisfies the linear recurrence relation
\[
X_{\delta} X^i_j= X^i_{j-b} + X^i_{j+b}
\]
for all $j \in \Z$ where $b$ is as in \ref{delta}.
\end{theo}
\begin{proof} This result follows immediately by \ref{delta} and \ref{extending}.
\end{proof}

Let $l\in Q_0$ be a vertex attached to an extending vertex $i$. Then
the projective indecomposable module associated with $l$ appears in
an Auslander-Reiten triangle
\[ P_i \to P_l \to \tau^{-1} P_i \to \Sigma  P_i .\]

This gives us the following relation between the frieze sequence associated
with the vertex $l$ and the sequence associated with the extending
vertex
\[ X^l_j= X^i_jX^i_{j+1}-1 \] for all $j \in \Z$. By
\ref{ProductSum}, the sequence $X^l_j$ satisfies a linear recursion
relation.

Let now $l \in Q_0$ be a vertex such that there is an oriented path
$i_0=s, \ldots , i_t=l$ from an extending vertex $s \in Q_0$ to
$l$ in $Q$ of length at least two. Then there are exchange triangles
of the form

\[P_s \to P_l  \to \tau^{-1} P_{i_{t-1}}  \to \Sigma P_s\]

and

\[ \tau^{-1} P_{i_{t-1}} \to \tau^{-2} P_{i_{t-2}} \to P_s \to
\Sigma \tau^{-1} P_{i_{t-1}}.
\]
This gives the following relation between the sequences associated
with the vertices appearing in the oriented path
$$ X^s_j X^{i_{t-1}}_{j-1}= X^l_j+ X^{i_{t-2}}_{j-2}$$ for all $j \in \Z$.
As all vertices connected to an extending vertex satisfy a linear
recurrence relation by the previous case, we can assume that the
sequences $X^{i_{t-1}}_j$ and $X^{i_{t-2}}_j$ satisfy a linear
recursion using induction on the path length. Then the sequence
$(X^l_j)_{j \in \Z}$ also satisfies a linear recursion relation. In
the case $\tilde E_6$, for every non-extending vertex $l$ of $Q$,
there is an extending vertex and an oriented path from the extending
vertex to $l$. Therefore, we obtain linear recurrence relations for
all vertices of the quiver $Q$ of type $\tilde E_6$.

In the case $\tilde E_7$, only the vertex labeled $7$ can not be
reached by an oriented path starting in an extending vertex. In this
case, we consider the exchange triangles

\[ P_1 \to \tau^{-1} P_7 \to \tau^{-4} P_4 \to \tau P_1 \]

and

\[ \tau^{-4} P_4 \to N \to P_1 \to \tau^{-3} P_4 \ko \]

where $N$ is the cokernel of any non-zero morphism $\tau^{-1} P_1 \to \tau^{-4} P_4$. Then
$\tau N$ is the cokernel of the map $P_1 \to \tau^{-3}P_4$.
It is the indecomposable regular simple module of dimension
vector $001100011$ which belongs to the mouth of the tube of width
$4$ (cf. the tables at the end of \cite{DlabRingel74a}).

For $ \tilde E_8$ we use an analogous method. The vertices $6$, $7$
and $8$ can not be reached by an oriented path starting in an
extending vertex. Therefore we consider the following exchange
triangles
\[ P_1 \to \tau^{-2} P_7 \to \tau^{-7}P_1 \to \tau P_1\]
and
\[ \tau^{-7} P_1 \to N \to P_1 \to \tau^{-6} P_1 \ko \]
where $N$ is the cokernel of any non-zero morphism
$\tau^{-1} P_1 \to \tau^{-7} P_1$.
It is the regular simple module of dimension vector $001111001$
which belongs to the mouth of the tube of width 5 by \cite[page
49]{DlabRingel74a}. Hence the sequence $X_{\tau^i N }$ is periodic.
These triangles give the relation $
X^7_j=X^1_{j+2}X^1_{j-5}-X_{\tau^{j-2}N}$, which proves that the
sequence at the vertex $7$ satisfies a linear recurrence relation.
The next exchange triangles are given by
\[ P_1 \to  \tau^{-1} P_6 \to \tau^{-3}P_7 \to \tau P_1 \]
and
\[ \tau^{-3}P_7 \to  \tau^{-8}P_1 \to P_1 \to \tau^{-2}P_1. \]    Here the relation is $X^6_j =X^1_{j+1}X^7_{j-2}-X^1_{j-7}$.
Finally, the last exchange triangles can be taken as
\[P_1 \to \tau^{-1} P_8 \to \tau^{-2} P_6 \to \tau P_1 \]
and
\[  \tau^{-2}P_6 \to \tau^{-4}P_7 \to P_1 \to \tau^{-1}P_6 .\]This gives the exchange relation $X^8_j =X^1_{j+1}X^6_{j-1}-X^7_{j-3}$. This proves that all frieze
sequences associated with vertices of the exceptional quivers satisfy linear recurrence relations.

\section{Type $\tilde A_{p,q}$}
We choose $p, q \in \N$ such that $ q \le p$ and use the same
labeling and orientation for $Q$ as in the proof of \ref{delta}. We view the
labels of vertices modulo $p+q$. Note that all vertices of $Q$ are extending vertices.

\begin{theo}\label{vectorsequence}
(a) For every vertex $i\in Q_0$ and every regular simple module
$M$ with dimension vector $\delta$, there are generalized exchange
triangles

\[ P_i \to \tau^{l_i} P_{i-q} \to M \to \Sigma P_i \]

and

\[ M \to \tau^{r_i} P_{i+q} \to P_i \to \Sigma M \ko\]
where $l_i=i-q$ for $ 0\le i \le q$ and $l_i=\mbox{max}\{q-i, -q\}$
for $q < i$ and $r_i=-l_{i+q}$.

(b) We obtain relations between the frieze sequences associated to the
vertices $i$, $i+q$ and $i-q$ of the form
\[ X_{\delta} X^i_j=X^{i+q}_{j+l_i}+X^{i-q}_{j+r_i} \] for all $i\in Q_0$ and $j \ge n$.
\end{theo}
\begin{proof}
Using \ref{existence} and \ref{relation} we obtain the existence
of the first  triangle. If we replace $i$ by $i+q$ in
the first triangle and perform a rotation, we obtain the triangle
\[ M \to P_{i+q} \to \tau^{l_{i+q} } P_i \to \Sigma M. \] Applying
$\tau^{-l_{i+q}}$ to this triangle gives the second triangle.
Exactly as in the proof of \ref{Dn} we can apply powers of $\tau$
to the generalized exchange triangles and we obtain the recurrence
relations stated.
\end{proof} If $p$ equals $q$ the relation from the previous Theorem yields \[ X_{\delta} X^i_j=X^{i+q}_{j+l_i}+X^{i+q}_{j+r_i} \] for all $i\in Q_0$ and $j \ge n$ as $i+q=i-q$ seen modulo $2q$.
If we iterate once, we obtain \[
(X_{\delta}^2-2)X^i_j=X^{i}_{j-q}+X^i_{j+q},\]
using the fact that $r_i-l_i=q$ for all $i\in Q_0$ and $j \ge q$.
Hence we can see immediately that all frieze sequences associated to
vertices of the quiver $Q$ of type $\tilde A_{q,q}$ satisfy a linear recurrence relation.
In the case $p \not = q$ we need a different argument. We consider
the sequence of vectors $(v(j))_{j\in \N}$ given by $v(j)=(X^0_j,
\ldots, X^n_j)$. Then by \ref{vectorsequence}, there are matrices
$A_0, \ldots , A_{n}$ such that $v(j+n+1) =\sum_{t=0}^{n} A_t
v(j+t)$ for all $j \in\N$. Using \ref{ProductSum} b), it is clear
that the frieze sequence associated with any vertex $i$ satisfies a linear
recurrence relation.

\section{Non simply laced types}
If $Q$ is a finite quiver without oriented cycles which is
endowed with a valuation (cf. \cite{DlabRingel74a}),
one can define frieze sequences in a natural way.
We refer to chapter~3, equation~(1) of
\cite{AssemReutenauerSmith09} for the exact definition and to the
appendix of \cite{AssemReutenauerSmith09} for the list of affine
Dynkin diagrams, which underlie the affined valued quivers.
As in section 7.3 of \cite{AssemReutenauerSmith09}, we can
obtain the linear recurrence relation for a frieze sequence associated with
a vertex of a  valued quiver of affine type from the linear
recurrence relation of a frieze sequence associated with the vertices of a non
valued affine quiver. This can be done using the folding technique.
An introduction to the folding technique can for example be found in
section 2.4 of \cite{FominZelevinsky03b}.
\begin{theo}
The frieze sequences associated with vertices of a quiver of the type $\tilde G_{21}$,
$\tilde G_{22}$, $\tilde F_{41}$, $\tilde F_{42}$,
$\tilde A_{11}$ or $\tilde A_{12}$ satisfy linear recurrence
relations.
\end{theo}
We obtain the linear recurrence relation for a frieze sequence associated
to a vertex of a quiver of type $ \tilde G_{22}$ or $\tilde  F_{42}$ by folding
$\tilde E_6$ using the obvious action by $\Z/3\Z$ respectively $\Z/2\Z$.
The linear recurrence relations in the case
$\tilde F_{41}$ are obtained by folding $\tilde E_7$ using a natural action by $\Z/2\Z$.
In the cases $\tilde G_{21}$, $\tilde A_{11}$ or $\tilde A_{12}$, they are obtained by
folding $\tilde D_4$ using actions of $\Z/3\Z$, $\Z/4\Z$ and $\Z/2\Z \times \Z/2\Z$
respectively.

\section{On the minimal linear recurrence relations}

For type $\tilde D_4$ (with the vertices numbered as in
the proof of Theorem~\ref{delta}), one checks that the
following are the polynomials of the minimal linear recurrence relations:
\[
\mbox{vertices $0$, $1$, $3$, $4$ : } \lambda^4 - X_\delta \lambda^2 +1 \ko\quad
\mbox{vertex $2$ : } (\lambda-1)(\lambda^2 - X_\delta \lambda +1) \ko
\]
where
\begin{align*}
X_\delta = &\frac{x_0^2 x_1^2 +2 x_0^2 x_1^2 x_2 + x_0^2 x_1^2 x_2^2 + 4
x_0 x_1 x_2 x_3 x_4  + 2 x_0 x_1 x_3 x_4  }{x_0 x_1 x_2^2 x_3 x_4} \\
 &+ \frac{x_2^2 x_3^2 x_4^2 +2 x_2 x_3^2 x_4^2+ x_3^2 x_4^2}{x_0 x_1 x_2^2 x_3 x_4} .
\end{align*}

Most of the recurrence relations one obtains from our proofs are
not minimal. However, we conjecture that for type $\tilde{A}$, they
are. In the following tables, for each vertex $i$ of
a quiver $Q$ of type $\tilde D$ or $\tilde E$, we exhibit a polynomial which
we conjecture to be associated with the minimal linear recurrence
relation for the frieze sequence $(X^i_j)_{j\in\N}$. Our conjecture
is based on the relations we have found and on numerical evidence
obtained using Maple. For an integer $d$ and an element $c$
of the field $\mathcal{F}=\Q(x_0, \ldots, x_n)$, where $n+1$ is the number of
vertices of $Q$, we put
\[
P(2d, c) = \lambda^{2d} - c \lambda^d + 1.
\]
The element $X_\delta$ of the field $\mathcal{F}$ is
always defined as after Lemma~\ref{lemma:x-delta}.
For type $\tilde{D}_n$, we number the vertices as in the proof
of Theorem~\ref{delta} and for the exceptional types as in
section~\ref{s:exceptional-types}.
For type $\tilde D_n$, where $n>4$ is even, we conjecture the
following minimal polynomials.
Notice that the polynomials for $\tilde D_4$ are not obtained by specializing
$n$ to $4$ in this table.
\begin{center}
\begin{tabular}{|c|c|l|} \hline
vertex & degree & polynomial \\ \hline
$0,1,n-1,n$ & $2n-4$&  $ P(2n-4, X_\delta)$ \\
$2, \ldots, n/2-1$ & $2n-4$ & $(\lambda^{n-2}-1) P(n-2, X_\delta) P(n-2, -X_\delta)$ \\
$n/2$ & $3n/2 -3$ & $(\lambda^{n/2-1}-1) P(n-2, X_\delta)$ \\ \hline
\end{tabular}
\end{center}
For type $\tilde D_n$, where $n>3$ is odd, we conjecture the
following minimal polynomials.
\begin{center}
\begin{tabular}{|c|c|l|} \hline
vertex  & degree & polynomial \\ \hline
$0,1,n-1,n$ & $4n-8$ & $P(4n-8,X_\delta)$ \\
$n/2$ & $2n-4$ & $(\lambda^{n-2}-1)P(n-2,X_\delta)$ \\ \hline
\end{tabular}
\end{center}
For $\tilde E_6$, we conjecture the following minimal polynomials.
\begin{center}
\begin{tabular}{|c|c|l|} \hline
vertex & degree & polynomial \\ \hline
$1,3,5$ & $12$ & $P(12,X_\delta)$ \\
$2,4,6$ & $9$ & $(\lambda^3-1) P(6,X_\delta)$ \\
$7$ & $ 16$ & $P(4,X_\delta) P(12, X_\delta)$ \\ \hline
\end{tabular}
\end{center}
For $\tilde E_7$, we conjecture the following minimal polynomials.
\begin{center}
\begin{tabular}{|c|c|l|} \hline
vertex & degree &  polynomial \\ \hline
$1,4$ & $24$ & $P(24, X_\delta)$ \\
$2,5$ & $36$ & $(\lambda^{12}-1) P(12, X_\delta) P(12,-X_\delta)$ \\
$3,6$ & $32$ & $P(24,X_\delta) P(8, X_\delta)$ \\
$7$ & $18$ & $(\lambda^6-1) P(12, X_\delta)$ \\
$8$ & $24$ & $(\lambda^6-1) P(12, X_\delta) P(6, -X_\delta)$ \\ \hline
\end{tabular}
\end{center}
For $\tilde E_8$, we conjecture the following minimal polynomials.
\begin{center}
\begin{tabular}{|c|c|l|} \hline
vertex & degree & polynomial \\ \hline
$1$ & $60$ & $P(60, X_\delta)$ \\
$2,7$ & $45$ & $(\lambda^{15}-1) P(30,X_\delta)$ \\
$3,6$ & $80$ & $P(60,X_\delta) P(20,X_\delta)$ \\
$4,8$ & $75$ & $(\lambda^{15}-1) P(30, X_\delta) P(30, X_\delta^2 -2)$ \\
$5$ & $132$ & $P(60,X_\delta) P(12, X_\delta) P(60, X_\delta^3 - 3 X_\delta)$ \\
$9$ & $85$ & $(\lambda^{15}-1) P(30,X_\delta) P(30, X_\delta^2-2) P(10,X_\delta)$ \\ \hline
\end{tabular}
\end{center}
Notice that for $\tilde E_6$ and $\tilde E_8$, the polynomial associated with
a vertex $i$ only depends on the coefficient $\delta_i$ of
the root $\delta$ but that the analogous statement for
 $\tilde E_7$ does not hold since the polynomials associated
 with the vertices $3$ and $6$ are different.



\def\cprime{$'$} \def\cprime{$'$}
\providecommand{\bysame}{\leavevmode\hbox to3em{\hrulefill}\thinspace}
\providecommand{\MR}{\relax\ifhmode\unskip\space\fi MR }
\providecommand{\MRhref}[2]{%
  \href{http://www.ams.org/mathscinet-getitem?mr=#1}{#2}
}
\providecommand{\href}[2]{#2}

\end{document}